\documentclass[12pt]{article}
\usepackage[margin=1.25in]{geometry}
\usepackage{amsmath,amsthm,amssymb}
\usepackage{graphicx}
\usepackage{wasysym}
\usepackage[linesnumbered,ruled,vlined]{algorithm2e}
\usepackage{extarrows}
\usepackage{hyperref}
\usepackage{tikz}

{\bfseries}{\rmfamily}
\newtheorem{theorem}{Theorem}
\newtheorem{definition}[theorem]{Definition}
\newtheorem{conjecture}{Conjecture}
\newtheorem{observation}[theorem]{Observation}

\newtheorem{lemma}[theorem]{Lemma}

\newcommand{\gen}[1]{\ensuremath{\langle #1\rangle}}
\usepackage{xcolor}
\definecolor{light-gray}{gray}{0.95}

\usepackage{array}
\newcolumntype{L}{>{$}l<{$}}

\global\long\def\Aut{\operatorname{Aut}}

\DeclareMathOperator*{\argmax}{arg\,max}
\DeclareMathOperator*{\argmin}{arg\,min}

\newcommand{\Z}{\mathbb{Z}} 
\newcommand{\calL}{\mathcal{L}} 
\newcommand{\calP}{\mathcal{P}} 

\newcommand{\calF}{\mathcal{F}}

\newcommand{\calB}{\mathcal{B}}
\newcommand{\Supp}{\mathrm{Supp}}

\newcommand\sq{\mathbin{\text{\scalebox{.84}{$\square$}}}}

\title{Bilevel Programming for Pebbling Numbers of Lemke Graph Products}
\author{
Jonad Pulaj\\
\texttt{jopulaj@davidson.edu}\\
Davidson College
\and
Kenan Wood\thanks{Corresponding author}\\
\texttt{kewood@davidson.edu}\\
Davidson College
\and
Carl Yerger\\
\texttt{cayerger@davidson.edu}\\
Davidson College
}
\date{}

\begin{document}

\maketitle

\begin{abstract}
    Given a configuration of indistinguishable pebbles on the vertices of a graph, a pebbling move consists of removing two pebbles from one vertex and placing one pebble on an adjacent vertex. The \emph{pebbling number} of a graph is the least integer such that any configuration with that many pebbles and any target vertex, some sequence of pebbling moves can place a pebble on the target. Graham's conjecture asserts that the pebbling number of the cartesian product of two graphs is at most the product of the two graphs' pebbling numbers. Products of so-called \emph{Lemke graphs} are widely thought to be the most likely counterexamples to Graham's conjecture, provided one exists.

    In this paper, we introduce a novel framework for computing pebbling numbers using bilevel optimization. We use this approach to algorithmically show that the pebbling numbers of all products of 8-vertex Lemke graphs are consistent with Graham's conjecture, with the added assumption that pebbles can only be placed on a set of at most four vertices.
\end{abstract}

\section{Introduction}\label{Sec:intro}
Graph pebbling consists of two-player games on graphs that involve extremal configurations of indistinguishable pebbles on the vertices that are subject to \emph{pebbling moves}, where two pebbles are removed from one vertex and one pebble is placed on an adjacent vertex. In the pebbling game, one player selects some small number of pebbles, the adversary responds by distributing the selected number of pebbles onto the vertices of a graph and choosing a target vertex, and the first player applies a sequence of pebbling moves to place a pebble to the target vertex; the first player wins if he succeeds at getting a pebble to the target. 
The \emph{pebbling number} of a graph is the minimum number of pebbles such that the first player always has a winning sequence of pebbling moves, regardless of the adversary's actions.

This topic originated from the following zero-sum problem in number theory: for any set of $n$ integers, is there
always a finite subset $S$ whose sum is $0$ modulo $n$ that satisfies $\sum_{s \in S} \gcd(s, n) \le n$?
Kleitman and Lemke \cite{addition-theorem-mod-n} showed that the answer to this question is ``yes,'' and Chung \cite{chung} later showed the equivalence between Kleitman and Lemke's technique and graph pebbling.
However, pebbling has since become a rich source of interesting graph theoretic questions with a well-established literature. We refer the reader to Hurlbert \cite{hurlbert2013graph} for an excellent survey on the topic. 

Graph pebbling can also be interpreted as a network optimization problem in which resources (fuel, energy) are used up in transport: a pebbling move consumes one unit of supply (a pebble) from the total in order to move the supply across an edge. 
Thus the problem becomes to determine the least amount of resources so that it is always possible to move anywhere else in the network, regardless of the initial distribution of resources. This fundamental network flow problem is equivalent to the notion of the pebbling number of a graph.

\subsection{Notation and Terminology}
We use the following notational conventions throughout this paper. Unless otherwise stated, all graphs are simple and connected, without loops or multiple edges.
A \emph{pebbling configuration} (or \emph{distribution}) on a graph $G$ is a map $p: V(G) \to \Z_{\ge 0}$; intuitively, for $v \in V(G)$, $p(v)$ is the number of pebbles on vertex $v$.
A \emph{pebbling move} on $p$ consists of removing two pebbles from one vertex and placing one an adjacent vertex; if $p'$ is the resulting configuration, we often write $p \to p'$. Sequences of pebbling moves starting from an initial configuration are often denoted by the corresponding sequence of directed edges traversed.
The \emph{support} of $p$, denoted $\Supp(p)$, is the set $\{v \in V(G): p(v) > 0\}$.
The \emph{size} of $p$, denoted $|p|$, is defined as $\sum_{v \in V(G)} p(v)$.

If for some $r\in V(G)$, there exists a sequence of pebbling moves $p = p_0 \to p_1 \to \cdots \to p_k$ such that $p_k(r) \ge 1$, we say that $p$ is \emph{$r$-solvable}; otherwise, $p$ is \emph{$r$-unsolvable}. 
Given any root vertex $r \in V(G)$, the \emph{$r$-rooted pebbling number} of $G$, denoted $\pi(G, r)$, is the least integer $m$ such that any pebbling configuration of size $m$ is $r$-solvable. We define the \emph{pebbling number} of $G$ as $\pi(G) = \max_{r \in V(G)} \pi(G, r)$. 

\subsection{Context and Graham's Conjecture}
Of significant interest in graph pebbling is a well-established conjecture of Graham that traces back to the origins of the subject. If $G$ and $H$ are graphs, we denote the Cartesian product of $G$ and $H$ by $G \sq H$, defined with vertex set $V(G) \times V(H)$ and edge set 
\[
\{\{(u,v), (x,y)\} \subseteq V(G) \times V(H): (u = x \wedge vy \in E(H)) \vee (v = y \wedge ux \in E(G)\}.
\]
\begin{conjecture}
    [Graham \cite{chung}] If $G$ and $H$ are graphs, then $\pi(G \sq H) \le \pi(G)\pi(H)$.
\end{conjecture}
Graham's conjecture has been proven for products of paths \cite{chung}, products of cycles \cite{products-of-cycles-1, Snevily2000}, products of fan and wheel graphs \cite{Snevily2000}, and some products of graphs in which one has the so-called \emph{2-pebbling property} \cite{products-of-cycles-1,Snevily2000,Wang2009}.

One of the major obstacles towards progress on Graham's conjecture is the lack of computationally tractable tools.
Computing pebbling numbers is extremely computationally difficult; Milans and Clark \cite{milans-clark} shows that deciding the language $\{\gen{G, k}: \pi(G) \le k\}$ is $\Pi_2^\mathtt{P}$-complete.  Watson \cite{Watson2005} shows that even determining if a configuration of pebbles if $r$-solvable is \texttt{NP}-complete. This computational intractability appears to be a major barrier, since we are unable to find pebbling numbers and extract patterns for moderately large graphs. Nonetheless, there have been recent efforts towards designing more efficient algorithms and techniques to bound pebbling numbers. Cusack, Green, and Powers \cite{lemke-classification-9} give algorithms for pebbling numbers and related problems, classifying all 8- and 9-vertex graphs without the \emph{2-pebbling property} described below. Hurlbert \cite{Hurlbert2010}, Cranston et al. \cite{Cranston2017}, Kenter and Skipper \cite{Kenter2018}, and Flocco, Pulaj, and Yerger \cite{Flocco-Pulaj-Yerger} have used techniques from discrete linear optimization for graph pebbling, yielding strong bounds. Despite this line of computational exploration, pebbling is still largely intractable. 

However, many classes of graphs are particularly well-behaved with respect to pebbling. A graph $G$ satisfies the \emph{2-pebbling property} if for every configuration $p$ with size at least $2\pi(G)-|\Supp(p)|+1$, there is a sequence of pebbling moves starting at $p$ that yields at least 2 pebbles on any specified root vertex. These graphs have nice properties with respect to Graham's conjecture; for example, if $G$ has the 2-pebbling property and $T$ is a tree, then $\pi(G \sq T) \le \pi(G) \pi(T)$ \cite{Snevily2000}.

Graphs that do not posses the 2-pebbling property are called \emph{Lemke graphs}. Although certain small Lemke graphs behave well with trees and complete graphs, as shown by Gao and Yin \cite{Gao-Yin-2017-Lemke}, little is known about general Lemke graphs. Towards enumeration, Cusack, Green, and Powers show that the smallest Lemke graph is on 8 vertices, and there are 3 distinct \emph{minimal} Lemke graphs on 8 vertices \cite{lemke-classification-9}. It is not difficult to show that all of these Lemke graphs on 8 vertices are \emph{Class 0}, where a graph is Class 0 if its pebbling number equals it number of vertices. 

\subsection{Contributions}\label{Subsec:conributions}
Given the lack of understanding of Lemke graphs in relation to Graham's conjecture, it is generally believed that products of Lemke graphs are the most likely counterexample if one exists. As such, the 8-vertex Lemke graphs have attracted a large amount of attention, compared to other graphs \cite{lemke-classification-9, Flocco-Pulaj-Yerger, Hurlbert2010,Kenter2018,kenter-computing-bounds}. In this paper, we add to this line of inquiry. The main result in this paper suggests strong evidence that all products of 8-vertex Lemke graphs are consistent with Graham's conjecture.

The specific contributions of this paper are two-fold.
\begin{enumerate}
    \item We give a formulation of computing pebbling numbers via \emph{bilevel integer linear programming}.
    To our knowledge, this is the first method in graph pebbling that uses techniques from bilevel programming. Our method appears to be (experimentally) much more efficient than any other known algorithm for computing pebbling numbers.
    \item We use this bilevel framework to prove that Graham's conjecture holds for all products of 8-vertex Lemke graphs, when the supports of pebbling configurations are restricted to at most 4 vertices (formalized in Theorem \ref{Thm:main}). While all known results on these graphs focus on bounding the pebbling number---which does not provide much evidence towards consistency with Graham's conjecture---this result provides substantial evidence that, indeed, $\pi(L' \sq L'') \le \pi(L') \pi(L'')$ for all 8-vertex Lemke graphs $L'$ and $L''$.
\end{enumerate}

Our implementation\footnote{\href{https://github.com/KenanWood/Bilevel-Programming-for-Pebbling-Numbers}{https://github.com/KenanWood/Bilevel-Programming-for-Pebbling-Numbers}} uses the bilevel solver from \cite{bilevel-solver-github,bilevel-solver}.
This bilevel mixed-integer linear programming solver is state-of-the-art and makes use of structural cuts that perform very well in practice in comparison to other known implementations.
Experimentally, our bilevel programming approach together with the implementation lends itself particularly well to candidate Class 0 graphs, which describes the products of 8-vertex Lemke graphs considered here. 

Let us now formally state our main result, Theorem \ref{Thm:main}, after stating the following definitions.
For a subset $S \subseteq V(G)$ and a root $r \in V(G)$, define $\pi_S(G, r)$ to be the minimum number $m$ such that any distribution of size $m$ with a support that is a subset of $S$ is $r$-solvable. For $k \ge 1$, we define 
\[
\pi_k(G, r) = \max_{S \in \binom{V(G)}{k}} \pi_S(G, r)
\]
and 
\[
\pi_k(G) = \max_{r \in V(G)} \pi_k(G, r).
\]
That is, $\pi_k(G)$ is the least integer such that all pebbling configurations with support of size at most $k$ having size $\pi_k(G)$ can reach any target vertex via a sequence of pebbling moves.
The quantity $\pi_k(G, r)$ is called the \emph{$r$-rooted support-$k$ pebbling number} of $G$, and $\pi_k(G)$ is called the \emph{support-$k$ pebbling number} of $G$.

Let $\calL$ be the set of all Lemke graphs on 8 vertices, up to isomorphism. Note that $|\calL| = 22$, and there are exactly three minimal Lemke graphs in $\calL$ \cite{lemke-classification-9}; all 8-vertex Lemke graphs contain one of the three minimal ones. 
\begin{theorem}\label{Thm:main}
    For all $L', L'' \in \calL$, we have $\pi_4(L' \sq L'') \le \pi(L') \pi(L'')$.
\end{theorem}

The rest of this paper is organized as follows. In Section \ref{Sec:bilevel}, we introduce the necessary bilevel programming prerequisites. Section \ref{Sec:pebbling-formulation} describes the bilevel formulation of the pebbling game, specifically for computing pebbling numbers. Section \ref{Sec:Alg} describes the algorithm we use to compute support-$k$ pebbling numbers (Algorithm \ref{Alg:covering}), utilizing the symmetry of product graphs (Subsection \ref{Subsec:symmetry}) and advantages of covering designs (Subsection \ref{Subsec:coverings}). Also in Section \ref{Sec:Alg}, we prove Theorem \ref{Thm:main}. In Section \ref{Sec:computations}, we elaborate on the computational details in the proof of Theorem \ref{Thm:main} and discuss future directions towards formal verification of the results. Finally, concluding remarks may be found in Section \ref{Sec:conclusion}.







\section{Bilevel Programming}\label{Sec:bilevel}
In this section, we give a brief overview of bilevel programming and some formal definitions used throughout the rest of this paper.

Bilevel programming is a class of optimization problems, in which a \emph{leader} and a \emph{follower} are both trying to optimize their respective (potentially conflicting) objective functions \cite{bilevel-2,bilevel-3,bilevel-solver}; we refer the reader to \cite{Kleinert2021-survey} for a comprehensive overview of bilevel programming. Specifically, the leader makes a choice for the \emph{upper-level} variables, and the follower responds by making a choice for the rest of the variables (\emph{lower-level}) that satisfies all of the lower-level constraints and optimizes the follower's objective function over the lower-level constraints. In this paper, we assume the \emph{optimistic} model for the follower's decision making, where the follower will choose an optimal solution for the follower subproblem (many optimal solutions often exist) that most benefits the leader, in terms of satisfying upper-level constraints and improving the leader's objective.
This is in contrast to the \emph{pessimistic} model, where follower acts adversarially against the leader \cite{bilevel-optimistic-pessimistic}.

Note the similarity between these bilevel programming games and the pebbling game. In a certain ``dual'' sense to the pebbling game described in the introduction, the leader can be seen as the one placing the pebbles and choosing the root vertex, attempting to maximize the number of pebbles used, while the follower attempts to make a sequence of pebbling moves to maximize the final number of pebbles on the root vertex; an unsolvable configuration of pebbles is then characterized by a feasible solution for the leader such that the optimal solution for the follower achieves no pebbles on the root. This connection underlies the approach described in this paper.

Now we make bilevel integer linear programming precise. 
A \emph{bilevel integer linear program} (BILP) is an optimization problem of the form
\[
\begin{aligned}
    \min_{x \in \mathbb{Z}^{n_u}, y \in \mathbb{Z}^{n_l}} & \quad c_1^\top x + c_2^\top y \\
    \text{s.t.} & \quad A_1 x + A_2 y \leq b_1, \\
               & \quad y \in \argmin_{\bar{y} \in \mathbb{Z}^{n_l}} \{ d^\top \bar{y} \;|\; A_3 x + A_4 \bar{y} \leq b_2 \},
\end{aligned}
\]
where $x \in \Z^{n_u}$ is the vector of upper-level ``leader'' variables, $y \in \Z^{n_l}$ is the vector of lower-level ``follower'' variables, and $c_1, c_2, A_1, A_2, b_1, d, A_3, A_4, b_2$ are real vectors and matrices of appropriate sizes. Of course, either objective can be a maximization or minimization, and constraints need not be exclusively $\le$ constraints, but all BILPs can be expressed in this form. Note that the ``optimistic'' paradigm of bilevel programming is captured directly in this formal definition.

In this BILP, $c_1^\top x + c_2^\top y$ is the \emph{leader objective function}, $d^\top \bar{y}$ is the \emph{follower objective function}, $A_1 x + A_2 y \leq b_1$ is the \emph{upper-level constraints}, and $A_3 x + A_4 \bar{y} \leq b_2$ is the \emph{lower-level constraints}.

A \emph{feasible solution} is a pair $(x,y) \in \Z^{n_u} \times \Z^{n_l}$ such that all of the upper-level and lower-level constraints are satisfied, and the follower objective function is optimized. A BILP is \emph{feasible} if it has some feasible solution, and \emph{infeasible} otherwise. An \emph{optimal solution} for a BILP is a feasible solution that optimizes the leader objective function over all such feasible solutions. The \emph{optimal objective value} of the above BILP is the quantity $c_1^\top x + c_2^\top y$, where $(x,y)$ is any optimal solution; note that this term refers only to the objective value of the leader, and does not consider the follower.

\section{Bilevel Formulation for Pebbling}\label{Sec:pebbling-formulation}
Now we are ready to describe the bilevel program to compute pebbling numbers of graphs, extending the intuition discussed in the previous section. Specifically, we describe the bilevel program used to compute $\pi_S(G, r)$ for a graph $G$, root $r$, and support $S$.
Let us first present an integer linear program to test if a configuration of pebbles $p$ on a graph $G$ is $r$-solvable.

Define the \emph{arc set} of a graph $G$ by $A(G) = \{(u,v) \in V(G) \times V(G): \{u,v\} \in E(G)\}$. Given $v \in V(G)$, define $\delta^-(v) = \{a \in A(G): \exists u \in V(G), a = (u,v)\}$ and $\delta^+(v) = \{a \in A(G): \exists u \in V(G), a = (v,u)\}$.
Let $p$ be a pebbling configuration on $G$.
Define the following integer linear program, denoted $SOL_p(G, r)$, on integer variables $z_a$ for each $a \in A(G)$.
\begin{align*}
    \text{maximize } &\sum_{a \in \delta^-(r)} z_a\\
    \text{subject to } & \sum_{a \in \delta^+(v)} 2z_a \le p(v) + \sum_{a \in \delta^-(v)} z_a & \forall v \in V(G)\\
    & \sum_{a \in \delta^+(r)} z_a \le 0\\
    & z_a \in \Z_{\ge 0} & \forall a \in A(G)
\end{align*}

We use the following definition and known result in the proof of Lemma \ref{Lem:lower-level-equiv}.
\begin{theorem}[Acyclic Orderability Theorem \cite{milans-clark, pebbling-graphs}]
    If $D$ is an acyclic directed multigraph and $p:V(D) \to \Z_{\ge 0}$ is a pebbling configuration on $D$, then the following conditions are equivalent:
    \begin{enumerate}
        \item (Orderable). Some ordering of the edges of $D$ yields a valid sequence of pebbling moves starting from $p$.
        \item (Balance Condition). Every $v \in V(D)$ satisfies $p(v) + \deg_D^-(v) - 2 \deg_D^+(v) \ge 0$.
    \end{enumerate}
\end{theorem}

Throughout the rest of this section, we assume that the specified root is not in the support of the pebbling configurations. It suffices to consider only these cases since
$\pi_k(G, r) = \max_{r \notin S \in \binom{V(G)}{k}} \pi_S(G, r)$ for all graphs $G$, roots $r \in V(G)$, and $1 \le k \le |V(G)|-1$.

\begin{lemma}\label{Lem:lower-level-equiv}
    If $r \in V(G)$ and $p$ is a pebbling configuration on $G$ with $r \notin \Supp(p)$, then $p$ is $r$-solvable if and only if the optimal objective value of $SOL_p(G, r)$ is nonzero.
\end{lemma}
\begin{proof}
    Suppose $p$ is $r$-solvable.
    Let $(u_1, v_1), (u_2, v_2), \dots, (u_t, v_t)$ be a sequence of directed edges such that pebbling moves starting at $p$ along these edges results in a configurations $p'$ in which $p'(r)$ is maximized (so $p'(r) > 0$); this sequence is nonempty because $r \notin \Supp(p)$. Let $D$ be a directed multigraph with vertex set $V(G)$ and edge multiset $\{(u_1, v_1), \dots, (u_t, v_t)\}$. Without loss of generality, we may assume that $D$ is acyclic. Since $D$ is acyclic and orderable under $p$ (with ordering $(u_1, v_1), (u_2, v_2), \dots, (u_t, v_t)$), the Acyclic Orderability Theorem implies that $p(v) + \deg^-_D(v) - 2\deg^+_D(v) \ge 0$. Thus if we define $z_a$ to be the number of times a pebble is sent across $a \in A(G)$, it follows that $(z_a)_{a \in A(G)}$ is a feasible solution of $SOL_p(G, r)$ with objective value equal to $p'(r)$.

    Conversely, let $z \in \Z_{\ge 0}^{A(G)}$ be an optimal solution of $SOL_p(G, r)$. Suppose that $\sum_{a \in \delta^-(r)} z_a > 0$. Construct a directed multigraph $D$ by vertex set $V(G)$ and edge set $\{a \in A(G): z_a > 0\}$, where the multiplicity of each $a \in A(G)$ is exactly $z_a$. Note that $\deg^+(r) = 0$ since $\sum_{a \in \delta^+(r)} z_a \le 0$ (which holds by feasibility). Thus $r$ and all of the arcs in $\delta^-(r)$ are not in any directed cycles of $D$.
    Consider iteratively removing the edges in any directed cycles in $D$ until no more are left; let the resulting sequence of directed multigraphs be $D = D_1, D_2, \dots, D_t = D'$. For each $i \in [t]$ and $uv \in A(G)$, let $z^i_{uv}$ be the multiplicity of edge $uv$ in $D_i$ (possibly $0$). Observe that for all $v \in V(G)$, we have $p(v) + \deg_{D_1}^-(v) - 2\deg_{D_1}^+(v) = p(v) + \sum_{a \in \delta^-(v)} z^1(a) - \sum_{a \in \delta^+(v)} 2z^1(a) \ge 0$ since $z^1 = z$ and $z$ is feasible. Suppose $i \in [t]$ with $i \ge 2$ and $v \in V(G)$. If $v$ is not a vertex in the $(i-1)$st deleted cycle, then we have $p(v) + \deg_{D_i}^-(v) - 2\deg_{D_i}^+(v) = p(v) + \deg_{D_{i-1}}^-(v) - 2\deg_{D_{i-1}}^+(v)$; otherwise, exactly one edge in each of $\delta^-(v)$ and $\delta^+(v)$ is deleted, so that
    \begin{align*}
        p(v) + \deg_{D_i}^-(v) - 2\deg_{D_i}^+(v) &= p(v) + (\deg_{D_{i-1}}^-(v) - 1) - 2(\deg_{D_{i-1}}^+(v) - 1)\\
        &\ge p(v) + \deg_{D_{i-1}}^-(v) - 2\deg_{D_{i-1}}^+(v).
    \end{align*}
    It follows by induction that $p(v) + \deg_{D_i}^-(v) - 2\deg_{D_i}^+(v) \ge 0$ for all $i \in [t]$ and $v \in V(G)$. In particular, this holds for $D' = D_t$. Since $D'$ is acyclic by construction, $D'$ is orderable by the Acyclic Orderability Theorem, so that some total ordering of the arcs in $D'$ is a valid sequence of pebbling moves. Since none of the arcs in $\delta^-(r)$ are in a cycle of $D = D_1$, these arcs are never deleted in the construction of $D'$, so that $\deg_{D'}^-(r) = \deg_{D}^-(r) = \sum_{a \in \delta^-(r)} z_a > 0$. This shows that in the sequence of pebbling moves given by the orderability of $D'$ under $p$, some pebble is moved across an arc in $\delta^-(r)$. Hence $p$ is $r$-solvable.
\end{proof}

We are now able to construct a bilevel integer linear program to compute $\pi_S(G, r)$ for any $r \notin S \subseteq V(G)$. In the following definition, we abuse notation slightly, by writing $(y_v)_{v \in S}$ for the vector $\bar{y}$ indexed by $V(G)$ such that $\bar{y}_v = y_v$ for $v \in S$ and $\bar{y}_v = 0$ for $v \notin S$.

\begin{definition}
    Let $G$ be a graph, $r \in V(G)$, and $S \subseteq V(G)$ with $r \notin S$. Define $PI_S(G, r)$ as
    \begin{align}
        \text{maximize } &\sum_{v \in S} y_v\\
        \text{subject to } & y_v \in \Z_{\ge 0} & \forall v \in S\\
        & \sum_{v \in S} y_v \ge |V(G)|\\
        & x \in \argmax ( SOL_{(y_v)_{v \in S}}(G, r) )\\
        & \sum_{a \in \delta^-(r)} x_a \le 0
    \end{align}
\end{definition}
In the above definition, each variable $y_v$ represents the number of initial pebbles on vertex $v$, so that $(y_v)_{v \in S}$ represents an initial pebbling configuration. The variables $x \in \Z_{\ge 0}^{A(G)}$ represent an optimal flow of pebbles to the root vertex $r$, which is enforced by (4). Constraint (5) requires that the final number of pebbles on the root is 0 in this optimal flow. Thus feasible solutions $(x, y)$ correspond to $r$-unsolvable initial configurations whose support is a subset of $S$. Constraint (3) is the Class 0 infeasibility cut described in the introduction. Finally, the objective function signifies that we are computing the maximum size of an $r$-unsolvable configuration. 

Observe that if it is known that $L \le \pi_S(G, r) - 1 \le U$ for some lower and upper bounds $L$ and $U$ (recall that $\pi_S(G, r) - 1$ is the maximum size of an $r$-unsolvable configuration with support a subset of $S$), then we may add the cuts $\sum_{v \in S} y_v \ge L$ and $\sum_{v \in S} y_v \le U$; in particular, we may use any known upper bounds on $\pi(G)$ or $\pi(G, r)$.

\begin{theorem}\label{Thm:bilevel-equivalence}
    Let $r \in V(G)$ and $S \subseteq V(G)$ with $r \notin S$.
    If $PI_S(G, r)$ is infeasible, then $\pi_S(G, r) \le |V(G)|$. Otherwise, the optimal objective value of $PI_S(G, r)$ is precisely $\pi_S(G, r)-1$.
\end{theorem}
\begin{proof}
    Since feasible solutions correspond to $r$-unsolvable configurations (by Lemma \ref{Lem:lower-level-equiv}), if $PI_S(G, r)$ is infeasible, then constraint (3) implies that there are no $r$-unsolvable configurations with support that is a subset of $S$ having size at least $|V(G)|$, showing that $\pi_S(G, r) \le |V(G)|$.
    If $PI_S(G, r)$ is feasible, then
    the optimal objective value of $PI_S(G, r)$ is precisely the largest size of an $r$-unsolvable configuration with support that is a subset of $S$, which is $\pi_S(G, r) - 1$.
\end{proof}

\section{Support-\textit{k} Pebbling Number Algorithm}\label{Sec:Alg}
In this section, we describe the algorithm for computing $\pi_k(G)$ for a graph $G$ and support size $k$, which uses the bilevel programming approach described in the previous section at its core. We conclude this section with the proof of Theorem \ref{Thm:main}.

\subsection{Symmetric Computations}\label{Subsec:symmetry}
The na\"ive way of computing $\pi_k(G) = \max_{r \in V(G)} \pi_k(G, r)$ is to compute $\pi_k(G, r)$ for all possible roots $r$, and take the maximum. However, we use the following observations as a computational optimization. 
\begin{observation}\label{Obs:root-isomorphism}
    Let $r_1, r_2 \in V(G)$.
    If there is an automorphism $\phi \in \Aut(G)$ such that $\phi(r_1) = r_2$, then $\pi_k(G, r_1) = \pi_k(G, r_2)$.
\end{observation}
Hence we only need to compute $\pi_k(G, r)$ for a root $r$ in each automorphism orbit of $G$. 

Recall that $\calL$ is the set of 8-vertex Lemke graphs. Let $L = L_1 \in \calL$ be the following minimal Lemke graph in Figure \ref{fig:lemke}, which was the first one discovered \cite{chung}, and hence it has received more attention than the others. We focus on the square of this graph to get a sense for the reduced computational work, though the same algorithm outlined below can be used for the others as well.

\begin{figure}[ht]
\centering
\begin{tikzpicture}[scale=2, every node/.style={circle, inner sep=2pt}]

\node[shape=circle,draw=black] (1) at (0,0) {$v_1$};
\node[shape=circle,draw=black] (2) at (1,0) {$v_2$};
\node[shape=circle,draw=black] (3) at (1,1) {$v_3$};
\node[shape=circle,draw=black] (4) at (1.5, 2) {$v_4$};
\node[shape=circle,draw=black] (5) at (.5,3) {$v_5$};
\node[shape=circle,draw=black] (6) at (-.5,2) {$v_6$};
\node[shape=circle,draw=black] (7) at (0,1) {$v_7$};
\node[shape=circle,draw=black] (8) at (.5, 1.9) {$v_8$};

\draw (1) -- (2);
\draw (2) -- (3);
\draw (3) -- (4);
\draw (4) -- (5);
\draw (5) -- (6);
\draw (6) -- (7);
\draw (7) -- (8);
\draw (3) -- (6);
\draw (3) -- (8);
\draw (4) -- (7);
\draw (5) -- (7);
\draw (5) -- (8);
\draw (1) -- (7);

\end{tikzpicture}
\caption{The original Lemke graph $L = L_1$.}
\label{fig:lemke}
\end{figure}
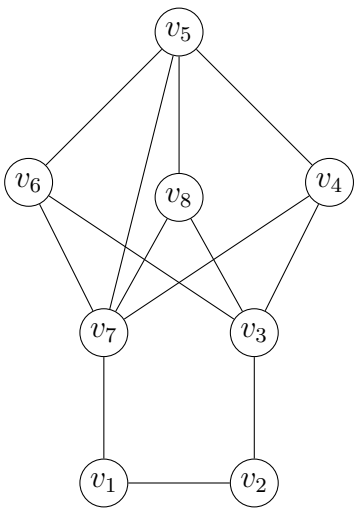

Obtaining graph automorphism orbits from \texttt{SageMath} \cite{sage}, for $L \sq L$, we obtain a reduction from checking 64 roots, down to 21 roots, partly because of the ``slice'' symmetry of product graphs.

Next, we use symmetry to reduce the number of bilevel integer programs needed to run for computing $\pi_k(G, r)$. Na\"ively, we must compute $\pi_S(G, r)$ for all $S \subseteq V(G) \setminus \{r\}$ with $|S| = k$, resulting in $\binom{|V(G)|-1}{k}$ runs of $PI_S(G, r)$. We can reduce this number by applying the following observation.
\begin{observation}\label{Obs:support-isomorphism}
    Let $S_1, S_2 \in \binom{V(G) \setminus \{r\}}{k}$. If there is an automorphism $\phi \in \Aut(G)$ such that $\phi(S_1) = S_2$ and $\phi(r) = r$, then $\pi_{S_1}(G, r) = \pi_{S_2}(G, r)$.
\end{observation}


Using both of these optimizations, we have to run $1,880,808$ different bilevel integer programs to compute $\pi_4(L \sq L)$. Compared to the na\"ive method of computing $|V(L \sq L)| \cdot \binom{|V(L \sq L)| - 1}{4} = 38,122,560$ different bilevel programs, these optimizations alone result in a speedup factor of 20.27 times.


\subsection{Covering Designs}\label{Subsec:coverings}
In our actual implementation, we integrate symmetry with the notion of a \emph{covering design}. Classically, a $(n,k,t)$-covering design is a family $\mathcal{B} \subseteq \binom{[n]}{k}$ such that every $t$-element set $S \subseteq [n]$ is contained within one member-set of $\mathcal{B}$ \cite{covering-designs-2, covering-designs-1, covering-designs-3}. We use a weakening of this concept, only requiring that the ``non-isomorphic'' sets are contained in a set in $\mathcal{B}$.

\begin{definition}
    If $n \ge k \ge 1$ and $\calF \subseteq \calP(X)$, where $X$ is finite, then a \emph{$(X,k,\mathcal{F})$-covering design} is a family $\mathcal{B} \subseteq \binom{X}{\le k}$ such that every $S \in \calF$ is contained within one member-set of $\mathcal{B}$.
\end{definition}

Let $r \in V(G)$. For $S, T \in \binom{V(G) \setminus \{r\}}{k}$, we write $S \cong_{G,r,k} T$ if there exists an automorphism of $G$ that fixes $r$ and maps $S$ onto $T$.
Let $\calF(G, r, k) \subseteq \binom{V(G) \setminus \{r\}}{k}$ be a family of representatives of the equivalence classes under the equivalence relation $\cong_{G,r,k}$; see Observation \ref{Obs:support-isomorphism}.

Let $c \ge k$ be the size of the covering sets.
Our algorithm consists of greedily choosing the member-sets of a family $\mathcal{B}$ by arbitrarily (and maximally) choosing sets of the uncovered sets in $\calF$ so that their union has size at most $c$, which will be added as a member-set of $\calB$; this algorithm is repeated until all the sets of $\calF$ are covered by sets in $\calB$. 
This algorithm results in $\calB$ being a $(V(G)\setminus\{r\}, c, \calF(G, r, k))$-covering design. Specifically, the algorithm is formally described in Algorithm \ref{Alg:covering}. 

\SetKwInput{KwInput}{Input}
\SetKwInput{KwOutput}{Output}

\begin{algorithm}[!ht]
\DontPrintSemicolon
\caption{$\calB(G, r, k; c)$}
\label{Alg:covering}
  
\KwInput{A graph $G$, $r \in V(G)$, $k \in [|V(G)|]$, and $c \ge k$.}
\KwOutput{A $(V(G) \setminus \{r\}, c, \calF(G, r, k))$-covering design.}

$\calF \gets \calF(G, r, k)$

$\calB \gets \emptyset$

\While{$\calF \ne \emptyset$}{

    $S \gets \emptyset$

    \For{$T \in \calF$}{
        
        \If{$|S \cup T| \le c$}{
            $S \gets S \cup T$
        }
    }

    $\calF \gets \{A \in \calF: A \not\subseteq S\}$

    $\calB \gets \calB \cup \{S\}$
}

\Return $\calB$

\end{algorithm}

It should be obvious that Algorithm \ref{Alg:covering} always terminates, since the first set $T$ in line 5 in each iteration of the \textbf{while} loop satisfies $|S \cup T| = |T| = k \le c$, and so this set is contained in the final $S$, showing this set $T$ is removed from $\calF$ in line 8. Now consider $\calB$ in its terminating state. The conditional in line 6 shows that every $S \in \calB$ satisfies $|S| \le c$; since $\calF(G, r, k) \subseteq \binom{V(G) \setminus \{r\}}{k}$, we further have $\calB \subseteq \binom{V(G) \setminus \{r\}}{\le c}$. Finally, for every $T \in \calF(G, r, k)$, since the terminating condition specifies that $\calF = \emptyset$, $T$ is deleted from $\calF$ in some iteration; in this iteration, we must have $T \subseteq S$ by line 8, and $S \in \calB$ by line 9; thus every member of $\calF(G, r, k)$ is covered by a set in $\calB(G, r, k; c)$. It follows that $\calB(G, r, k;c)$ is a $(V(G) \setminus \{r\}, c, \calF(G, r, k))$-covering design.

Thus, instead of computing
$
\pi_k(G, r)
$
with $\max_{S \in \binom{V(G)\setminus \{r\}}{k}} \pi_S(G, r)$ or $\max_{S \in \calF} \pi_S(G, r)$, we use the fact that $\pi_S(G, r) \le \pi_T(G, r)$ when $S \subseteq T$ to obtain the bound
\[
\pi_k(G, r) \le \max_{S \in \calB(G, r, k; c)} \pi_S(G, r).
\]

Using this approach, we obtain our main result, Theorem \ref{Thm:main}.

\begin{proof}[Proof of Theorem \ref{Thm:main}]
    Let $L_1, L_2, L_3 \in \calL$ be the three minimal Lemke graphs on 8 vertices. Since $\pi_4(G) \le \pi_4(H)$ for $G \supseteq H$ with $V(G) = V(H)$, and because each of the three minimal 8-vertex Lemke graph is Class 0 (easy to prove, but can also be obtained using our framework immediately), it follows that the pebbling number of every 8-vertex Lemke graph is 8. Thus it suffices to show that $\pi_4(L' \sq L'') \le 8^2 = 64$ for all $L' = L_i$ and $L'' = L_j$ for $1 \le i \le j \le 3$.

    On each of these six product graphs $G = L' \sq L''$ and each nonisomorphic root $r \in V(G)$ (in the context of Observation \ref{Obs:root-isomorphism}), we verify that $PI_S(G, r)$ is infeasible for all $S \in \calB(G, r, 4; 8)$, discussed in Section \ref{Sec:computations}. Thus, by Theorem \ref{Thm:bilevel-equivalence}, we have $\pi_S(G, r) \le 64$ for all $S \in \calB(G, r, 4; 8)$. Since $\calB(G, r, 4; 8)$ is a $(V(G)\setminus \{r\}, 8, \calF(G, r, 4)$-covering design, every $T \in \calF(G, r, 4)$ is contained in some member $S \in \calB(G, r, 4; 8)$, so that $\pi_T(G, r) \le \pi_S(G, r) \le 64$. By Observation \ref{Obs:support-isomorphism} and the definition of $\calF(G, r, 4)$, it follows that $\pi_4(G, r) \le 64$, and that $\pi_4(G) \le 64$, as desired.
\end{proof}

\section{Computational Discussion}\label{Sec:computations}
In this section, we describe more details regarding the computations at the heart of the proof of Theorem \ref{Thm:main}. 

We parallelize the computation in a straightforward way that is sufficient for our purposes. We use seven distinct machines in parallel, four using an Intel Xeon Processor E5-2620 v4 with 32 CPUs, each running at 2.1 GHz, and 128 GB of memory, while three use an Intel(R) Xeon(R) Gold 6226R with 64 CPUs, each running at 2.9 GHz, with 384 GB of memory; each of the seven systems
have two NUMA nodes.
For each graph $L' \sq L''$ for minimal $L', L'' \in \calL$ as in the proof of Theorem \ref{Thm:main}, we use \texttt{SageMath} to obtain the nonisomorphic roots (the orbits of the automorphism group), and distribute the roots equitably among the seven machines.

Once the work has been dispatched to the machines, we compute the bilevel integer program $PI_S(L' \sq L'', r)$ over all the roots $r$ for the given machine, and all nonisomorphic supports $S$. Some cases may take much longer than others, so we cap the computation time for each instance at 30 minutes before continuing onto the next instance. This resulted in exactly 37 instances that were abandoned throughout the computation of the support-4 pebbling number of all six graphs. For each of these 37 instances, we changed the leader's objective sense to minimization instead of maximization (which is equivalent in terms of feasibility) and reran the instances, which performed very well in these more extreme scenarios. Somewhat surprisingly, changing the objective sense in this way did not seem to affect most cases, only these more extremal instances.

Table \ref{Table:times} summarizes, for each of the six graphs, the number of nonisomorphic roots (which equals the number of automorphism orbits, $|V(G)/\Aut(G)|$), the total number of instances over all nonisomorphic roots ($\sum_{r} |\calB(G, r, 4, 8)|$), the average time to solve each bilevel integer programming instance (denoted $t_{\mathrm{avg}}$), and the total time to complete the computation if done on a single machine (denoted $t_{\mathrm{total}} = t_{\mathrm{avg}} \cdot \sum_{r} |\calB(G, r, 4, 8)|$).

\begin{table}[ht]
\centering
\caption{Summary of computation metrics for the six graphs.}
\label{Table:times}
\begin{tabular}{|c|c|c|c|c|c|}
\hline
\textbf{Graph $G$} & \textbf{$|V(G)/\Aut(G)|$} & \textbf{$\sum_{r} |\calB(G, r, 4, 8)|$} & \textbf{$t_{\mathrm{avg}}$ (s)} & \textbf{$t_{\mathrm{total}}$ (s)} & \textbf{$t_{\mathrm{total}}$ (hrs)}\\ \hline
$L_1 \sq L_1$ & 21 & 121512 & 1.302 & 158192 & 43.94\\ \hline
$L_1 \sq L_2$ & 42 & 435726 & 1.094 & 476630 & 132.40\\ \hline
$L_1 \sq L_3$ & 24 & 226757 & 1.187 & 269146 & 74.76\\ \hline
$L_2 \sq L_2$ & 28 & 417149 & 1.144 & 477169 & 132.55\\ \hline
$L_2 \sq L_3$ & 28 & 421731 & 1.205 & 508041 & 141.12\\ \hline
$L_3 \sq L_3$ & 10 & 115039 & 1.616 & 185857 & 51.63\\ \hline
\end{tabular}
\end{table}

The complete time data from our computations and the code used in the experiments can be found in \cite{GitHub-Wood}.

We also tested our algorithm on a few other graphs as well. For example, we ran the algorithm on the 3-cube, which terminated in 0.233 seconds, and the 4-cube, which terminated in 3.327 seconds. We are also able to run some supports of size 20 for the graph $L_1 \sq L_1$, terminating in less than 2 hours; however, other cases were not able to terminate in the allotted time of 24 hours. Our method was not able to handle the 24-vertex 4th weak Bruhat graph \cite{Hurlbert2010}, nor the 32-vertex 5-cube in the span of several days. Thus it appears that our algorithm is limited by supports of around 20 vertices. 
This algorithm is substantially better than the previous best method for computing pebbling numbers by Cusack, Green, and Powers \cite{lemke-classification-9}, although their method is more reliable, using purely symbolic computation.

It is important to note that the results presented in this paper could be prone to floating point arithmetic error. The bilevel solver we use in this paper is based on CPLEX \cite{cplex}, which is a widely used mixed-integer linear programming (MILP) solver; CPLEX is designed for primarily industry uses, where a small degree of error is not paramount, so that that the speed advantage of floating point arithmetic outweighs the potential error. To our knowledge, there are no bilevel integer linear program solvers that use exact rational arithmetic, hence our choice of solver. 

Nonetheless, there is a foreseeable avenue for verification of our pebbling results. The most promising is to implement known algorithms for solving bilevel integer programs \cite{bilevel-2, bilevel-3, bilevel-solver} using an exact MILP solver such as exact SCIP \cite{exact-SCIP-1,exact-SCIP-2} instead of CPLEX for black-box MILP computations. Exact SCIP is a state-of-the-art exact rational MILP solver. Upon solving a MILP instance, exact SCIP can also produce a certificate of correctness, called a VIPR certificate \cite{vipr}. Thus, a bilevel integer programming solver could be implemented by using exact SCIP to solve each MILP instance; for further verification, we could additionally employ Satisfiability Modulo Theory (SMT) solvers such as cvc5 to verify each of the resulting VIPR certificates, using the transformation described by Wood et al. \cite{smt-for-vipr}. Other similar methods of utilizing SMT solvers to verify MILP instances could also be used \cite{local-configurations}.

Another avenue towards further verification of the results presented in this paper is to use quantified constraint satisfaction problem (QCSP) solvers, which are exact methods of solving some two-player games. There are known methods of expressing general (nonlinear) bilevel programming problems as QCSP instances \cite{bilevel-to-qcsp}, which could then be passed to a QCSP solver such as Quacode \cite{quacode-paper}. 
Exploring any of these paths towards verification would make for interesting future work.

\section{Conclusion}\label{Sec:conclusion}
In this paper, we develop a framework for computing pebbling numbers of graphs using techniques from bilevel optimization in order to overcome a well-established computational intractability. We use this framework, together with symmetry and covering designs, to show that all products of 8-vertex Lemke graphs have a support-4 pebbling number of at most 64, providing significant evidence that Graham's conjecture does indeed hold on these ``most likely counterexamples.''

There are a few avenues for further exploration. Is it possible to partially exploit the 2-pebbling property to reduce computational work, as is done in \cite{kenter-computing-bounds}? We believe that the bilevel program $PI_S(G,r)$ can be modified slightly to handle cases where the support is \emph{at least} a certain large size. Thus, it would be insightful to obtain a nontrivial upper bound on the support size of a counterexample to Graham's conjecture for any of the Lemke graph products considered in this paper, provided that a counterexample exists. Exploring either of these possibilities further would pave the way for a better understanding of graph pebbling.

\bibliographystyle{plain}
\bibliography{references}

\end{document}